\documentclass[10pt]{article}
\font\smallit=cmti10

\usepackage{amsthm,amssymb,amsmath,xcolor} 

\definecolor{aquamarine}{rgb}{0.5, 1.0, 0.83}

\makeatletter

\renewcommand\section{\@startsection {section}{1}{\z@}
{-30pt \@plus -1ex \@minus -.2ex}
{2.3ex \@plus.2ex}
{\normalfont\normalsize\bfseries\boldmath}}

\renewcommand\subsection{\@startsection{subsection}{2}{\z@}
{-3.25ex\@plus -1ex \@minus -.2ex}
{1.5ex \@plus .2ex}
{\normalfont\normalsize\bfseries\boldmath}}

\renewcommand{\@seccntformat}[1]{\csname the#1\endcsname. }

\makeatother

\newtheorem{lemma}{Lemma}
\newtheorem{proposition}{Proposition}
\newtheorem{corollary}{Corollary}

\begin{document}

\begin{center}
\uppercase{\bf A problem on concatenated integers}
\vskip 20pt
{\bf Josep M. Brunat}\\
{\smallit Retired from the Departament de Matem\`atiques, Universitat Polit\`ecnica de Catalunya, Barcelona.}\\
{\tt Josep.M.Brunat@upc.edu}\\ 
\vskip 10pt
{\bf Joan-C. Lario}\\
{\smallit Department de Matem\`atiques, Universitat Polit\`ecnica de Catalunya, Barcelona.}\\
{\tt Joan.Carles.Lario@upc.edu}\\ 
\end{center}
\vskip 30pt

\centerline{\bf Abstract}
Motivated by a WhattsApp message, we find out the integers $x> y\ge 1$ such that  $(x+1)/(y+1)=(x\circ(y+1))/(y\circ (x+1))$, where $\circ$ means the concatenation of the strings of two natural numbers (for instance $783\circ 56=78356$). The discussion  involves  the equation  $x(x+1)=10y(y+1)$, a slight variation of  Pell's equation related to the arithmetic of the Dedekind ring $\mathbb{Z}[\sqrt{10}]$. We obtain the infinite sequence $\mathcal{S}=\{(x_n,y_n)\}_{n\ge 1}$ of all the solutions of  the equation $x(x+1)=10y(y+1)$, which tourn out to have limit $1/\sqrt{10}$. The solutions of the initial problem on concatenated integers form the infinite subsequence of $\mathcal{S}$ formed by the pairs $(x_n,y_n)$ such that $x_n$ has one more digit that $y_n$.

\vskip 30pt

\section{Introduction}

A friend sent to us a WhatsApp message with the content of Figure~\ref{WA}; no words at all, only a sequence of numbers and mathematical symbols.

\begin{figure}[htb]
	\centering
	\setlength{\fboxrule}{1pt} 
	\fcolorbox{black}{aquamarine}{ 
		$
		\renewcommand{\arraystretch}{2.2}
		\begin{array}{l}
		\displaystyle{\frac{20!\,7!}{6!\,21!}=\frac{207}{621}=\frac{1}{3}=0.3333333333...}\\[8pt]
		\displaystyle{\frac{175!\,56!}{55!\,176!}=\frac{17556}{55176}=\frac{7}{22}=0.3181818181...}\\[8pt]
		\displaystyle{\frac{1500!\,475!}{474!\,1501!}=\frac{1500475}{4741501}=\frac{25}{79}=0.3164556962...}\\[8pt]
		\displaystyle{\frac{29600!\,9361!}{9360!\,29601!}=\frac{296009361}{936029601}=\frac{37}{117}=0.3162393162...}\\[8pt]
		\displaystyle{\frac{253075!\,80030!}{80029!\,253076!}=\frac{25307580030}{80029253076}=\frac{265}{838}=0.3162291169...}\\[8pt]
		\displaystyle{\frac{1124039!\,355453!}{355452!\,1124040!}=\frac{1124039355453}{3554521124040}=\frac{721}{2280}=0.3162280701...}\\[8pt]
		\displaystyle{\frac{2163720!\,684229!}{684228!\,21637221!}=\frac{2163720684229}{6842282163721}=\frac{949}{3001}=0.3162279240...}\\[8pt]
		\displaystyle{\frac{1}{\sqrt{10}}=0.3162277660\ldots} \\[8pt]
		\end{array}
		$
	}
	\caption{The motivating  WhatsApp message.}
	\label{WA}
\end{figure}

Letting $\circ$ denote the action of concatenate the strings of two natural numbers, for instance 
$783\circ 56=78356$, the first 
equalities in every row are of the form
$$
\frac{x!\, (y+1)!}{y!\,(x+1!)}=\frac{x\circ(y+1)}{y\circ (x+1)}
$$
with $x\ge y$. Simplifying the factorials, the above 
equality is equivalent to
\begin{equation}
\label{original}
\frac{y+1}{x+1}=\frac{x\circ(y+1)}{y\circ (x+1)}.
\end{equation}
The real numbers at the end of each row of the message along with the last row suggest that these fractions are decreasing and tend to the limit $1/\sqrt{10}$ as $x,y$ grow.

\vskip 0.35truecm

Our aim here is to solve the following two questions:
\begin{enumerate}
	\item[1)] Find out all pairs of integers $(x,y)$ with $x> y\ge 1$ verifying (\ref{original}).
	
	\item[2)] Show that there exist infinitely many solutions and that the fractions $(y+1)/(x+1)$, with $x> y\ge 1$ solutions of (\ref{original}), are decreasing as $x,y$ grow and that its limit  exists and equals $1/\sqrt{10}$.
\end{enumerate}

Observe that the case $y=0$ is not interesting, since then
$x\circ(y+1)=10x+1$ and the equality (\ref{original}) reads 
$$
\frac{1}{x+1}=\frac{10x+1}{x+1},
$$
that is fulfilled only when $x=0$. The case $x=y$ does not have any interest either, since the  equality~(\ref{original}) is a tautology and the quotients $(y+1)/(x+1)=(x+1)/(x+1)=1$ do not have limit $1/\sqrt{10}$. Therefore, we shall restrict ourselves to the cases $x>y\ge 1$. 

Formally, the concatenation of the strings of two positive integers runs as follows.
Let $a$ and $b$ be two positive integers, written with base-10 number system:
\begin{align*}
a&=a_r10^r+a_{r-1}10^{r-1}+\cdots +a_110+a_0 \quad (a_r\ne 0),\\
b&=b_s10^s+b_{s-1}10^{s-1}+\cdots +b_110+b_0, \quad (b_s\ne 0).
\end{align*}
Then, one has
\begin{align*}
a\circ b&=10^{s+1}a+b \\
&=
a_r10^{r+s+1}+\cdots +a_010^{s+1}+b_s10^s+b_{s-1}10^{s-1}+\cdots +b_110+b_0.
\end{align*}

Hereafter, as usual, for a positive real number $x$, the functions 
$\lfloor x \rfloor$, $\lceil x \rceil$, $\lfloor x \rceil$ denote the floor, ceiling, and round part of $x$. For future reference, we introduce the function $\delta(x)=\lfloor \log x \rfloor$. Note that $\delta(x)+1$ counts the number of digits of the integer $\lfloor x \rfloor$ with the 10-base system.

\vskip 0.3truecm

The next lemma transforms our WhatsApp question it into a diophantine equation plus a condition on the digits function $\delta$. 

\begin{lemma}
	\label{t=1}
	Let $x>y\ge 1$ be integers.  Then, one has
	\begin{equation}
	\label{main}
	\frac{y+1}{x+1}=\frac{x\circ(y+1)}{y\circ(x+1)}
	\end{equation}
	if, and only if, it holds $x(x+1)=10y(y+1)$ and $\delta(x+1)=\delta(y+1)+1$.
\end{lemma}
\begin{proof}
	Let denote $r=\delta(x+1)$, and $s=\delta(y+1)$.
	Assume that
	$$
	\frac{y+1}{x+1}=\frac{x\circ(y+1)}{y\circ (x+1)}=\frac{10^{s+1}x+y+1}{10^{r+1}y+x+1}.
	$$
	By cross-multiplying and simplifying, one gets
	$$
	x(x+1)=10^{r-s}y(y+1).
	$$
	Since $x>y$, one has $r-s\ge 1$. 
	Let 
	$
	v=\lfloor \log(x(x+1))\rfloor=\lfloor\log(10^{r-s}y(y+1))\rfloor.
	$
	We have $v\in\{2r-1,2r,2r+1\}$. The value $v=2r-1$ occurs only when $x+1=10^r$. In this case, we have 
	$
	x(x+1)=(10^r-1)10^r=10^{r-s}y(y+1)
	$ 
	so that 
	$$
	(10^r-1)10^s=y(y+1)\,.
	$$
	Since $10^{s}-1\le y < 10^{s+1}-1$, we have 
	$$(10^{r}-1)10^s= y (y+1) < (10^{s+1}-1)10^{s+1}$$
	and $10^r-1<(10^{s+1}-1)10$. 
	Since $r-s\geq 1$, we must have $r=s+1$.
	Hence, we have
	\begin{align*}
	&v=\lfloor \log(x(x+1))\rfloor\in\{2r,\ 2r+1\},\\[2pt]
	&v=\lfloor\log(10^{r-s}y(y+1))\rfloor\in\{r+s-1,\ r+s,\ r+s+1\},
	\end{align*}
	and $2r\le v\le r+s+1$. It follows $r-s\le 1$.
	Summarizing, we must have  $10y(y+1)=x(x+1)$ and
	$r=s+1$ as wanted. 
	
	Reciprocally, it is easy to check that if $10y(y+1)=x(x+1)$ and $r=s+1$, then equation~(\ref{main}) holds. 
\end{proof}

Thus, our first task is to determine the set of solutions:
$$
\mathcal{S}=\{(x,y)\in\mathbb{Z}^2: x>y\ge 1,\ x(x+1)=10y(y+1)\}\,,
$$
and then select the subset of those that concatenate well with regard to our WhatApp challenge:
$$
\mathcal{C}=\{(x,y)\in\mathcal{S}: \delta(x+1)=\delta(y+1)+1\}.
$$

\section{The associate Pell's equation}
\label{sPell}

In this section, we deal with the set $\mathcal{S}$ of positive solutions of the  diophantine equation
$$
x(x+1)=10y(y+1)\,.
$$ 
Multiplying by 4 and completing the square, the equation is equivalent to 
$$
(2x+1)^2-10(2y+1)^2=-9.
$$
Letting $a=2x+1$ and $b=2y+1$, the above equation reads  
\begin{equation}
\label{ab}
a^2-10\,b^2=-9.
\end{equation}
Therefore, we need to find all odd positive integer solutions of (\ref{ab}).  

To this end, we shall use the arithmetic properties of the ring $\mathbb{Z}[\sqrt{10}]$. 
It is a Dedekind domain; that is, every non-zero proper ideal factors into a product of prime ideals, and such a factorization is unique up to the order of the factors. We shall use some standard 
terminology and basic properties:
the conjugate of $\alpha=a+b\sqrt{10}\in\mathbb{Z}[\sqrt{10}]$ is  $\overline{\alpha}=a-b\sqrt{10}$; the norm of $\alpha$ is ${\operatorname{N}}(\alpha)=\alpha\overline{\alpha}=a^2-10b^2$. The conjugation $\alpha\mapsto\overline{\alpha}$ is an automorphism of the ring $\mathbb{Z}[\sqrt{10}]$ and the norm map is multiplicative; that is,  $\operatorname{N}(\alpha\beta)={\operatorname{N}}(\alpha){\operatorname{N}}(\beta)$ for all $\alpha,\beta\in\mathbb{Z}[\sqrt{10}]$. An element $u\in\mathbb{Z}[\sqrt{10}]$ is a unit if, and only if, $|N(u)|=1$. 
Note that the units of $\mathbb{Z}[\sqrt{10}]$ are just the numbers $a+b\sqrt{10}$ where $(a,b)$ are the solutions of the Pell equations $a^2-10b^2=\pm 1$.
If $u\in \mathbb{Z}[\sqrt{10}]$ is a unit, then $u^{-1}=\overline{u}$. Dirichlet's unit theorem describes the structure of the group of units; in our case, $\varepsilon=3+\sqrt{10}$ is the fundamental unit and, if $\langle \varepsilon \rangle $ is the multiplicative group generated by $\varepsilon$, the group of units is  $\mathbb{Z}[\sqrt{10}]^* = 
\{ \pm 1\} \times \langle \varepsilon \rangle $. Notice that $\operatorname{N}(\varepsilon)=-1$. We refer the reader to \cite{NivZucMon,Ono,Stevenhagen} for further details.

The diophantine equation~(\ref{ab}) in $\mathbb{Z}$, gets transformed  into the norm equation  
$$
\operatorname{N}(a+b\sqrt{10}) = a^2-10 b^2 = -9 
$$
in the ring $\mathbb{Z}[\sqrt{10}]$.
Since the norm is multiplicative, two elements of $\mathbb{Z}[\sqrt{10}]$ with equal norm differ (multiplicatively) by a unit of norm $1$; that is, by an even power of $\varepsilon$.
Thus, all we need is to find the elements in  $\mathbb{Z}[\sqrt{10}]$ of 
norm~$-9$.

We shall denote by $\langle\alpha,\beta\rangle$ the ideal of $\mathbb{Z}[\sqrt{10}]$ generated by $\alpha$ and $\beta$. The class number of
$\mathbb{Z}[\sqrt{10}]$ is two, so that the square of every ideal is a principal ideal. Recall also that the norm of an ideal $\mathfrak{a}$
is the cardinal of the quotient ring  $\mathbb{Z}[\sqrt{10}]/\mathfrak{a}$, and the norm of ideals is multiplicative. Moreover, the norm of a principal ideal agrees with the absolute value of the norm of a generator of the ideal. 
We consider the ideals
$$
\mathfrak{P}_1=\langle 3,\ 1+\sqrt{10}\rangle,\quad \mathfrak{P}_2= \langle 3,\ 1-\sqrt{10}\rangle.
$$

\begin{lemma} With the above notations, it holds:
	\label{factor}
	\begin{itemize}
		\item[\rm (i)] $\langle 3\rangle=\mathfrak{P}_1\cdot \mathfrak{P}_2$;
		
		\item[\rm (ii)] $\mathfrak{P}_1$ and $\mathfrak{P}_2$ are prime ideals;
		
		\item[\rm (iii)] $\mathfrak{P}_1^2=\langle 1+\sqrt{10}\rangle$ and $\mathfrak{P}_2^2=\langle 1-\sqrt{10}\rangle$.
		
		\item[\rm (iv)] Moreover, $\langle a+b\sqrt{10}\rangle\cdot\langle a-b\sqrt{10}\rangle=\langle 9\rangle$ if and only if, for some unit $u\in \mathbb{Z}[10]^*$, one has
		$$
		a+b\sqrt{10}\in\{3u,\ (1-\sqrt{10})u,\  (1+\sqrt{10})u\}.
		$$
		
	\end{itemize}
\end{lemma}

\begin{proof} 
	(i) On the one hand, we have
	$$
	\mathfrak{P}_1 \mathfrak{P}_2=\langle 9,\ 3(1+\sqrt{10}),\ 3(1-\sqrt{10}),\ -9\rangle \subseteq\langle 3\rangle.
	$$
	On the other hand,  $3=9-3(1+\sqrt{10})-3(1-\sqrt{10})\in \mathfrak{P}_1 \mathfrak{P}_2$. Hence, we obtain the equality $\langle 3\rangle=\mathfrak{P}_1 \mathfrak{P}_2$. 
	
	(ii) A routine checking shows that the map  $f\colon \mathbb{Z}[\sqrt{10}]\to \mathbb{Z}/3\mathbb{Z}$ defined by $f(a+b\sqrt{10})=[a-b]$
	is a surjective ring homomorphism. Clearly $f(3)=[3]=[0]$ and $f(1+\sqrt{10})=[0]$. Thus, $\mathfrak{P}_1\subseteq\operatorname{Ker}f$.
	Also, if $a+b\sqrt{10}\in\operatorname{Ker}f$, then $[a-b]=[0]$. That implies  $a=b+3c$ for some integer $c$ and then
	$$
	a+b\sqrt{10}=b+3c+b\sqrt{10}=3c+b(1+\sqrt{10})\in \mathfrak{P}_1.
	$$
	Therefore, $\operatorname{Ker}f=\mathfrak{P}_1$. Since $\mathbb{Z}[\sqrt{10}]/\mathfrak{P}_1\simeq\mathbb{Z}/3\mathbb{Z}$ is a field, it follows that $\mathfrak{P}_1$ is a maximal ideal and, in particular, it is a prime ideal.
	Mutatis mutandis we see that $\mathfrak{P}_2$ is a prime ideal.  
	
	(iii) We have
	$$
	\mathfrak{P}_1^2=\langle 3,\ 1+\sqrt{10}\rangle^2=\langle 3^2,\ 3(1+\sqrt{10}),\ (1+\sqrt{10})^2\rangle\subseteq\langle 1+\sqrt{10}\rangle,
	$$
	and
	$
	1+\sqrt{10}= (-2)\cdot 9 -3(1+\sqrt{10})+2(1+\sqrt{10})^2\in \mathfrak{P}_1^2.
	$
	Thus, $\mathfrak{P}_1^2=\langle 1+\sqrt{10}\rangle$. Analogously, for the conjugate ideal we get $\mathfrak{P}_2^2=\langle 1-\sqrt{10}\rangle$. 
	
	(iv)
	From the equalities
	$$
	\langle a+b\sqrt{10}\rangle\cdot\langle a-b\sqrt{10}\rangle 
	=\langle 9\rangle 
	=\langle 3\rangle \cdot \langle 3\rangle
	=\mathfrak{P}_1^2 
	\mathfrak{P}_2^2\,,
	$$
	by using that $\mathbb{Z}[\sqrt{10}]$ is a Dedekind domain and the fact that the norm is multiplicative on ideals, it follows that $\langle a+b\sqrt{10}\rangle$ must be one among of the ideals 
	$$
	\mathfrak{P}_1\mathfrak{P}_2\,,\quad
	\mathfrak{P}_1^2\,,\quad 
	\mathfrak{P}_2^2\,.
	$$
	
	Case $\langle a+b\sqrt{10}\rangle=\mathfrak{P}_1 \mathfrak{P}_2=\langle 3\rangle$. Then, we get $a+b\sqrt{10}=3u$ for some unit~$u$.
	
	Case $\langle a+b\sqrt{10}\rangle=\mathfrak{P}_1^2=\langle 1+\sqrt{10}\rangle$.
	Then, we get $a+b\sqrt{10}=(1+\sqrt{10})u$ for some unit~$u$.
	
	Case $\langle a+b\sqrt{10}\rangle=\mathfrak{P}_2^2=\langle 1-\sqrt{10}\rangle$. Then, we get $a+b\sqrt{10}=(1-\sqrt{10})u$ for some unit $u$. 
\end{proof}

Finally, we can exhibit all positive integers solutions of $a^2-10\,b^2=-9$. For a quadratic
number $\alpha=c+d\,\sqrt{10}$ in $\mathbb{Z}[\sqrt{10}]$, we shall say that $(c,d)$ are the coordinates of $\alpha$. Recall that $\varepsilon=3+\sqrt{10}$ and has norm $-1$. 

\begin{proposition} The positive integer solutions $(a,b)$ of $a^2-10b^2=-9$
	are the coordinates of:
	\begin{itemize}
		\item[\rm (i)] $3\,\varepsilon^n$ for every odd positive integer $n$;
		\item[\rm (ii)] $-(1-\sqrt{10})\,\varepsilon^n$, for every even positive integer $n$; 
		\item[\rm (iii)] $(1+\sqrt{10})\,\varepsilon^n$, for every even positive integer $n$. 
	\end{itemize}
\end{proposition}

\begin{proof}
	First, we check that the quadratic numbers  of type (i), (ii) or (iii) give positive solutions of $a^2-10\,b^2=-9$.	
	
	(i) It is clear that $3\,\varepsilon^n$ has positive coordinates; since $n$ is odd, $N(3\,\varepsilon^n)=9\cdot(-1)=-9$.
	(ii) Observe that $-(1-\sqrt{10})\,\varepsilon^2=41+13\sqrt{10}$ has positive coordinates, and the successive multiplications by $\varepsilon^2=19+6\sqrt{10}$ also produce quadratic numbers with positive coordinates. Moreover, $n$ being even, one has $N(-(1-\sqrt{10})\,\varepsilon^n)=(-9)\cdot 1=-9$. (iii) Certainly,  $(1+\sqrt{10})\,\varepsilon^n$ has positive coordinates for every positive integer $n$; moreover, if $n$ is even, then $N((1+\sqrt{10})\,\varepsilon^n)=(-9)\cdot 1=-9$.
	
	Reciprocally, we claim that any positive solution of our equation must arise from one of these types (i), (ii) or (iii).
	
	Let $(a,b)$ denote a positive solution 
	of $a^2-10\,b^2=-9$. Then, $\langle a+b\sqrt{10}\rangle\cdot\langle a-b\sqrt{10}\rangle =\langle 9\rangle$ and, due to Lemma \ref{factor} (iv), we must have $$a+b\sqrt{10}\in\{3u,\ (1-\sqrt{10})u,\ (1+\sqrt{10})u\}$$ 
	with $u$ a unit of norm $-1$, $1$, and $1$, respectively. The units 
	of norm~$-1$ are precisely 
	$\pm \varepsilon^n$ for odd integers $n$. 
	Since $\varepsilon=3+\sqrt{10}>0$ 
	satisfies 
	$\varepsilon^{-1}=-\overline{\varepsilon}=-3+\sqrt{10}>0$, the units with positive coordinates of norm $-1$ are precisely 
	$\varepsilon^n$ for odd positive integers $n$. The claim follows.
\end{proof}

We have obtained that the positive solutions of 
$a^2-10\,b^2=-9$ are the coordinates of the quadratic numbers in the sequence $a_n+b_n\sqrt{10}$:
$$ 
\begin{array}{l@{ \quad}c@{ \quad}c}
3\varepsilon\,, &  -(1-\sqrt{10})\varepsilon^2 \,,& (1+\sqrt{10})\varepsilon^2 \,, \\
3\varepsilon^3 \,, &  -(1-\sqrt{10})\varepsilon^4 \,,& (1+\sqrt{10})\varepsilon^4 \,,\\
3\varepsilon^5 \,, & -(1-\sqrt{10})\varepsilon^6\,, & (1+\sqrt{10})\varepsilon^6 \,, \\
\vdots & \vdots  & \vdots 
\end{array}
$$
With $\varphi=\varepsilon^2=19+6\sqrt{{10}}$, each row above is of the form
$$
3\,\varepsilon\,\varphi^{n-1},\  -(1-\sqrt{10})\,\varphi^n, \ (1+\sqrt{10})\,\varphi^n, \qquad n\ge 1.
$$
The $(n+3)$th term 
$a_{n+3}+b_{n+3}\sqrt{10}$ is obtained from the term  $a_n+b_n\sqrt{10}$ by the formula
$$
a_{n+3}+b_{n+3}\sqrt{10}
=(a_n+b_n\sqrt{10})\varphi\\
=(19a_n+60b_n)+(6a_n+19b_n)\sqrt{10},
$$
with initial values
\begin{align*}
a_1+b_1\sqrt{10}&=3\,\varepsilon=9+3\sqrt{10},\\
a_2+b_2\sqrt{10}&=-(1-\sqrt{10})\,\varphi=41+13\sqrt{10},\\
a_3+b_3\sqrt{10}&=(1+\sqrt{10})\,\varphi=79+25\sqrt{10}.
\end{align*}
Therefore, the positive integer solutions of~(\ref{ab})
are the pairs $(a_n,b_n)$ obtained by the recurrence
\begin{equation}
\label{recurrence ab}
a_{n+3}=19a_n+60b_n,\qquad b_{n+3}=6a_n+19b_n
\end{equation}
with initial values
$$
(a_1,b_1)=(9,3),\quad (a_2,b_2) =(41,13),\quad (a_3,b_3) =(79,25).
$$

Recall that we are interested only on the odd positive solutions $(a_n,b_n)$. Pleasantly, notice that the initial values are formed by odd integers and 
that if $a_n$ and $b_n$ are odd, then $a_{n+3}$ and $b_{n+3}$ are odd too. Hence $a_n$ and $b_n$ are odd for all $n$. Rewriting $a_n=2x_n+1$ and $b_n=2y_n+1$, we finally get the following result. 

\begin{proposition}
	\label{recurrence}
	The set $\mathcal{S}$ 
	consists of couples $(x_n,y_n)$ given by the recursion with initial values
	$$
	(x_1,y_1)=(4,1),\quad (x_2,y_2)=(20,6),\quad (x_3,y_3)=(39,12),
	$$
	and, for $n\ge 1$,
	\begin{equation}
	\label{eqrecurrence}
	x_{n+3}=19x_n+60y_n+39, \qquad y_{n+3}=6x_n+19y_n+12\,.
	\end{equation}
\end{proposition}

We think of $\mathcal{S}$ as a set or as a sequence indistinctly.
Table~\ref{33first} shows the first terms of $\mathcal{S}$. We display in green color the solutions in $\mathcal{C}$; these are the solutions we are interested in; that is, those $(x_n,y_n)$ satisfying 
$x_n+1$ has one more digit than $y_n+1$ which correspond to the solutions of our problem originated by the WhatsApp message.

\begin{table}[htb]
	$$
	\setlength{\arraycolsep}{8pt}
	\begin{array}{ll}
	(4, 1), & \textcolor{green}{(20, 6)}, \\
	(39, 12), & \textcolor{green}{(175, 55)}, \\
	(779, 246),& \textcolor{green}{(1500, 474)},\\
	(6664, 2107),& \textcolor{green}{(29600, 9360)},\\
	(56979, 18018),&  \textcolor{green}{(253075, 80029)},\\
	\textcolor{green}{(1124039, 355452)},& \textcolor{green}{(2163720, 684228)},\\
	(9610204, 3039013),& (42683900, 13497834),\\
	(82164399, 25982664), &(364934695, 115402483),\\
	\textcolor{green}{(1620864179, 512562258)},& \textcolor{green}{(3120083460, 986657022)},\\
	\textcolor{green}{(13857908224, 4382255359)},& (61550154920, 19463867988),\\
	\textcolor{green}{(118481007099, 37466984190)},&  (526235577835, 166410301177),\\
	\textcolor{green}{(2337285022799, 739114421304)},& (4499158186320, 1422758742216),\\
	\textcolor{green}{(19983094049524, 6319209189385)}, &(88755280711460, 28066884141582),
	\end{array}
	$$
	\caption{The first terms of $\mathcal{S}$, in green those belonging in $\mathcal{C}$}
	\label{33first}
\end{table}

The above recurrence allows to obtain a direct formula for $x_n$ and $y_n$. 

\begin{proposition}
	\label{explicit}
	For $k\in\{1,2,3\}$, let
	$$
	A_k=\frac{1}{2}\left(\frac{x_k}{\sqrt{10}}+y_k+\frac{10+\sqrt{10}}{20}\right).
	$$
	Then, for $n\ge 1$,
	$$
	x_{3n+k}=\lfloor A_k\sqrt{10}\,\varphi^n\rfloor,\quad y_{3n+k}=\lfloor A_k\,\varphi^n\rfloor.
	$$
\end{proposition}
\begin{proof}
	For $n\geq 1$ and $k\in\{1,2,3\}$, we have
	$$
	\begin{bmatrix}
	x_{3n+k}\\
	y_{3n+k}\\
	1
	\end{bmatrix} = 
	\begin{bmatrix}
	19 & 60 & 39 \\
	6 & 19 & 12 \\
	0 & 0 & 1
	\end{bmatrix}^n \cdot
	\begin{bmatrix}
	x_{k} \\
	y_{k} \\
	1
	\end{bmatrix} =
	P 
	\begin{bmatrix}
	\varphi^n & 0 & 0 \\
	0 & \varphi^{-n} & 0 \\
	0 & 0 & 1
	\end{bmatrix}
	P^{-1}\cdot
	\begin{bmatrix} 
	x_{k} \\
	y_{k} \\
	1
	\end{bmatrix},
	$$
	where $P = \begin{bmatrix}
	\sqrt{10} & -\sqrt{10} & -1 \\
	1 & 1 & -1 \\
	0 & 0 & 2
	\end{bmatrix}$ and $\varphi=\varepsilon^2=19+6\sqrt{10}$. 
	Hence,
	if
	$$
	B_k=\frac{1}{2}\left(\frac{-x_k}{\sqrt{10}}+y_k+\frac{10-\sqrt{10}}{20}\right),
	$$
	we have
	\begin{align}
	x_{3n+k}&=A_k\sqrt{10}\,\varphi^n+B_k\sqrt{10}\,\varphi^{-n}-\frac{1}{2},\nonumber\\
	y_{3n+k}&=A_k\,\varphi^n+B_k\,\varphi^{-n}-\frac{1}{2}.  \label{explicita llarga}
	\end{align}
	From the fact that $|B_k\,\varphi^{-1}|<|B_k\sqrt{10}\,\varphi^{-1}|<1/2$, and $\varphi^{-1}<1$, we obtain 
	$$
	x_{3n+k}=\lfloor A_k\sqrt{10}\,\varphi^n-1/2\rceil,\quad
	y_{3n+k}=\lfloor A_k\,\varphi^n-1/2\rceil.
	$$
	Now, for $z$  a real number, $\lfloor z-1/2\rceil=\lfloor z\rfloor$. Thus,
	$$
	x_{3n+k}=\lfloor A_k\sqrt{10}\,\varphi^n\rfloor,\quad
	y_{3n+k}=\lfloor A_k\,\varphi^n\rfloor.
	$$
\end{proof}

\begin{proposition} 
	\label{limit} For $(x_n,y_n)\in\mathcal{S}$, it holds
	$$
	\lim_n\frac{y_n+1}{x_n+1}=\lim_n\frac{x_n}{y_n}=\frac{1}{\sqrt{10}}.
	$$
	Moreover, $y_n/x_n$ is strictly increasing and $(y_n+1)/(x_n+1)$ is strictly decreasing.
\end{proposition}
\begin{proof}	
	As $\varphi=19+6\sqrt{10}>1$, from the formulas~(\ref{explicita llarga}) it is clear that, for each $k\in\{1,2,3\}$, 
	$$
	\lim_n\frac{y_{3n+k}}{x_{3n+k}}=\frac{1}{\sqrt{10}}.
	$$
	Every pair $(x_n,y_n)\in \mathcal{S}$ is captured in one of the three sequences $(x_{3n+k},y_{3n+k})$, thus
	$$
	\lim_n\frac{y_n}{x_n}=\frac{1}{\sqrt{10}}.
	$$
	Moreover,
	$$
	\lim_n\frac{y_n+1}{x_n+1}=\lim_n\frac{y_n/x_n+1/x_n}{1+1/x_n}=\lim_n\frac{y_n}{x_n}=\frac{1}{\sqrt{10}}.
	$$
	
	To prove that $y_n/x_n$ is increasing, we use the recursive definition of $(x_n,y_n)$. 
	For $n\ge 1$, 
	the sign of $y_{n+1}/x_{n+1}-y_{n}/x_{n}$ is the sign of
	\begin{align*}
	\gamma_n=x_{n}y_{n+1}-x_{n+1}y_{n}.
	\end{align*}
	To show that $y_n/x_n$ is increasing it is sufficient to proof that $\gamma_n>0$ for all~$n$. 
	Now, $\gamma_1=4$, $\gamma_2=6$ and $\gamma_3=45$. For $n\ge 4$, 
	\begin{align*}
	\gamma_n=&\phantom{-}(19x_{n-3}+60y_{n-3}+39)(6x_{n-2}+19y_{n-2}+12)\\[3pt]
	&-(19x_{n-2}+60y_{n-2}+39)(6x_{n-3}+19y_{n-3}+12))\\[3pt]
	=&\, \gamma_{n-3}+6(x_{n-2}-x_{n-3})+21(y_{n-2}-y_{n-3}).
	\end{align*}
	The sequences $x_n$ and $y_n$ are increasing, and by induction hypothesis, $\gamma_{n-3}>0$. Therefore $\gamma_n>0$ and the sequence $y_n/x_n$ is strictly increasing.
	
	Now, let $e_n=1/\sqrt{10}-y_n/x_n$. Observe that $e_n-e_{n+1}>0$ for all $n$, and
	$$
	\frac{y_n+1}{x_n+1}-\frac{y_n}{x_n}
	=\frac{1}{x_n+1}\left(1-\frac{y_n}{x_n}\right)                
	=\frac{1}{x_n+1}\left(1+e_n-\frac{1}{\sqrt{10}}\right). 
	$$
	Therefore,
	\begin{align*}
	\frac{y_n+1}{x_n+1}-\frac{y_{n+1}+1}{x_{n+1}+1}
	=&\frac{y_n}{x_n}+\frac{1}{x_n+1}\left(1+e_n-\frac{1}{\sqrt{10}}\right) \\
	&-\frac{y_{n+1}}{x_{n+1}}-\frac{1}{x_{n+1}+1}\left(1+e_{n+1}-\frac{1}{\sqrt{10}}\right) \\
	&>\frac{y_n}{x_n}-\frac{y_{n+1}}{x_{n+1}}+\frac{1}{x_{n+1}+1}(e_n-e_{n+1})>0.
	\end{align*}
	Hence, $(y_n+1)/(x_n+1)$ is strictly decreasing. 
\end{proof}

\section{Infinite solutions}

If the set $\mathcal{C}$ is infinite, then the sequence $(y_n+1)/(y_n+1)$
with $(x_n,y_n)\in\mathcal{C}$ has the same limit as $(y_n+1)/(x_n+1)$ with
$(x_n,y_n)\in\mathcal{S}$, which is $1/\sqrt{10}$. Thus, it only remains to prove that the set $\mathcal{C}$ is infinite. To this end, we first show that if $(x,y)\in \mathcal{S}$, then
neither $x+1$ nor $y+1$ is a power of $10$.

\begin{proposition}
	\label{no power 10}
	If  $(x,y)\in\mathcal{S}$, then
	$x+1\ne 10^\beta$ and $y+1\ne 10^\beta$ for all $\beta\ge 1$. 
\end{proposition}
\begin{proof}
	The strategy of the proof is as follows. 
	For an integer $m\ge 2$, we consider the sequence obtained by taking the terms of the sequence $\mathcal{S}$ modulus $m$:
	$$
	{\mathcal S}_m = \{ (x_n \bmod m, y_n \bmod m) \colon \text{ for }
	(x_n,y_n) \in {\mathcal S} \}\,.
	$$
	Since $\mathcal{S}$ is defined by a recurrence, the sequences ${\mathcal S}_m$ are periodic for all $m$. We denote by $\ell(m)$ the length period of ${\mathcal S}_m$. Assuming that $x_n+1$ or $y_n+1$ are powers of 10, we choose appropriate $m$ to derive a contradiction.
	
	Assume $x_n+1=10^\beta$. By inspection, we can (and do) assume that 
	$n\geq 5$, which implies $x_n>1000$ and $\beta\geq 3$. We take the modulus $m=8$. Since $10^\beta\equiv 0\pmod{8}$, we have $x_n=10^\beta-1\equiv 7\pmod{8}$.
	The condition $10y_n(y_n+1)=x_n(x_n+1)$, becomes $2y_n(y_n+1)\equiv 6\cdot 7\equiv 2\pmod 8$, which has no solution. Therefore, $x_n+1\ne 10^\beta$.
	
	Assume $y_n+1=10^\beta$. First, we take the modulus $m=9$. The first 12 terms of ${\mathcal S}_9$ are
	\begin{align*}
	&(4, 1), (2, 6), (3, 3), (4, 1), (5, 3), (6, 6), 
	(4, 1), (8, 0), (0, 0), \\
	&(4, 1), (2, 6), (3, 3)\,,
	\end{align*}
	so that ${\mathcal S}_9$ has period $\ell(m)=9$. We must have $y_n=10^\beta-1\equiv 0 \pmod{9}$ and then $x_n(x_n+1)=10y_n(y_n+1)\equiv 0\pmod{9}$.
	This forces the two conditions: $x_n\equiv 0,8\pmod{9}$ and $n\equiv 0,8\pmod{9}$.
	
	Now, we take $m=10$. The 33 first terms of ${\mathcal S}_{10}$ are
	\begin{align*}
	&(4, 1), (0, 6), (9, 2), (5, 5), (9, 6), (0, 4), (4, 7), (0, 0), (9, 8), (5, 9),\\
	&(9, 2), (0, 8), (4, 3), (0, 4), (9, 4), (5, 3), (9, 8), (0, 2), (4, 9), (0, 8), \\
	&(9, 0), (5, 7), (9, 4), (0, 6), (4, 5), (0, 2), (9, 6), (5, 1), (9, 0), (0, 0), \\
	&(4, 1), (0, 6), (9, 2)\,,
	\end{align*}
	so that ${\mathcal S}_{10}$ has period $\ell(m)=30$. Thus, we have $y_n=10^\beta-1\equiv 9\pmod{10}$, $x_n(x_n+1)\equiv 0\pmod{10}$, and $x_n\equiv 0,4,5,9\pmod{10}$. The pairs $(0,9)$ and $(9,9)$ are not in ${\mathcal S}_{10}$. If $(4,9)$ is a $n$-term of ${\mathcal S}_{10}$ then $n\equiv 19\pmod{30}$, and if $(5,9)$ is a $n$-term of ${\mathcal S}_{10}$ then
	$n\equiv 10\pmod{30}$. 
	
	Gluing the information mod $m=9$ and mod $m=10$, we see that none of the four systems
	$$
	n\equiv {g}\pmod{9},\ g\in\{0,8\},\qquad n\equiv {h}\pmod{30},\ h\in\{10,19\}
	$$
	has a solution because in the four cases $\gcd(9,30)=3$ does not divide $g-h$. Therefore,  $y_n+1\ne 10^\beta$. 
\end{proof}

As a consequence of Proposition~\ref{no power 10}, we can precise the condition on digits of Lemma~\ref{t=1}. 

\begin{corollary}
	\label{coro}
	Let $(x,y)\in\mathcal{S}$. Then, one has:
	\begin{itemize}
		\item [\rm (i)] $\delta(x)=\delta(x+1)$ and $\delta(y)=\delta(y+1)$;
		\item [\rm (ii)] $\delta(x)=\delta(y)$ or $\delta(x)=\delta(y)+1$.
	\end{itemize}
	Moreover, $\mathcal{C}=\{(x,y)\in\mathcal{S}: \delta(x)=\delta(y)+1\}$.
\end{corollary}
\begin{proof}
	(i) If $\delta(x)=\beta$, then $10^\beta\le x<10^{\beta+1}$ and $x+1\le 10^{\beta+1}$. By Proposition~\ref{no power 10}, we have $x+1\ne 10^{\beta+1}$. Therefore,
	$x+1<10^{\beta+1}$ and we have $\delta(x+1)=\beta=\delta(x)$. Analogously, $\delta(y+1)=\delta(y)$.
	
	(ii) As in the proof of Lemma~\ref{t=1}, put $r=\delta(x+1)$ and $s=\delta(y+1)$. By~(i), $\delta(x)=r$ and $\delta(y)=s$. Then,
	since $(x,y)\in\mathcal{S}$, then $\delta(x(x+1))\in\{2r,2r+1\}$ and $\delta(y(y+1))\in \{2s,2s-1\}=\{2r-1,2r\}$. It follows
	$r=s$ or $r=s+1$, that is, $\delta(x)=\delta(y)$ or $\delta(x)=\delta(y)+1$.
	
	Finally, by (i) and Lemma~\ref{t=1}, we have $\mathcal{C}=\{(x,y)\in\mathcal{S}: \delta(x)=\delta(y)+1\}$.  
\end{proof}

Now, we can prove that the set $\mathcal{C}$ is infinite. In fact, we shall prove that
in each of the three sequences $(x_{3n+k},y_{3n+k})$ there are no more than two consecutive terms which are not in $\mathcal{C}$. The next two remarks will be useful.
\medskip

\noindent\textbf{Remark} 
The number $\varphi=\epsilon^2=19+6\sqrt{10}$ satisfies the conditions $37<\varphi<38$ and $1441<\varphi^2<1442$.
Thus, for each positive real number $x$, we have 
\begin{equation}
\label{fites g}
\delta(x)+1\le \delta(x\, \varphi )\le \delta(x)+2,\quad
\delta(x)+3\le \delta(x\, \varphi^2)\le \delta(x)+4.
\end{equation}	
\medskip

\noindent\textbf{Remark} Due to Proposition~\ref{explicit}, we know
that $x_{3n+k}=\lfloor A_k\sqrt{10}\,\varphi^n\rfloor$. Let $\delta(x_{3n+k})=\beta$. We have $A_k\sqrt{10}\,\varphi^n<1+x_{3n+k}\le 10^{\beta+1}$. By Proposition~\ref{no power 10}, we have $x_{3n+k}+1\ne 10^{\beta+1}$. Therefore, 
$$10^\beta \le x_{3n+k}<A_k\sqrt{10}\,\varphi^n<x_{3n+k}+1<10^{\beta+1}\,.$$ 
Thus, $\delta(x_{3n+k})=\delta(A_k\sqrt{10}\,\varphi^n)$. 
Analogously, one gets $\delta(y_{3n+k})=\delta(A_k\varphi^n)$.

\medskip

\begin{proposition}
	\label{C infinite}
	For $k\in\{1,2,3\}$, the sequence $(x_{3n+k},y_{3n+k})$ has not three consecutive terms in $\mathcal{S}\setminus\mathcal{C}$. In particular, the set $\mathcal{C}$ is infinite. 
\end{proposition}
\begin{proof}
	Consider one of the  three sequences; that is, fix $k$. Rename the index of the terms and assume that we have consecutive elements such that $(x_n,y_n)\in\mathcal{C}$ and 
	$(x_{n+1},y_{n+1}),(x_{n+2},y_{n+2})\in\mathcal{S}\setminus\mathcal{C}$. We want to prove that $(x_{n+3},y_{n+3})\in\mathcal{C}$.
	
	For every $m\geq1$, denote $p_m=A_k\sqrt{10}\,\varphi^m$ and $q_m=A_k\varphi^m$. By using the last Remark, we have
	$$
	\delta(p_n)=\delta(q_n)+1,\quad  \delta(p_{n+1})=\delta(q_{n+1}),\quad  \delta(p_{n+2})=\delta(q_{n+2}),
	$$
	and we want to prove $\delta(p_{n+3})=\delta(q_{n+3})+1$. Let $\beta=\delta(p_n)$. Then, $\delta(q_n)=\beta-1$. By~(\ref{fites g}), one has
	$$
	\delta(p_{n+1})=\delta(p_n \varphi)\in\{\beta+1,\beta+2\},\quad 
	\delta(q_{n+1})=\delta(q_n \varphi)\in\{\beta,\beta+1\}.
	$$
	Since $\delta(p_{n+1})=\delta(q_{n+1})$, we must have $\delta(p_{n+1})=\delta(q_{n+1})=\beta+1$.
	
	As for $(p_{n+2},q_{n+2})$, by~(\ref{fites g}), we have
	$$
	\delta(p_{n+2})=\delta(p_{n+1}\varphi)\in\{\beta+2,\beta+3\},\quad
	\delta(p_{n+2})=\delta(p_{n}\,\varphi^2)\in\{\beta+3,\beta+4\}.
	$$
	Hence, $\delta(p_{n+2})=\delta(q_{n+2})=\beta+3$.  
	
	As for  $(p_{n+3},q_{n+3})$,  by~(\ref{fites g}), we obtain
	$$
	\delta(p_{n+3})=\delta(p_{n+2}\varphi)\in\{\beta+4,\beta+5\},\ \quad 
	\delta(q_{n+3})=\delta(q_{n+2}\varphi)\in\{\beta+4,\beta+5\}.
	$$
	Now, on one hand we have, that the relation $q_n<10^\beta\le p_n$ and $\varphi<38$ implies 
	$$
	q_{n+3}=q_n\,\varphi^3<10^\beta\cdot 38^3,\quad \delta(q_{n+3})\le \delta(10^\beta\cdot 38^3)=\beta+4.
	$$
	Hence, $\delta(q_{n+3})=\beta+4$.
	On the other,
	$$
	p_{n+3}
	=q_{n+3}\sqrt{10}
	=q_{n+2}\,\varphi\,\sqrt{10}
	\ge 10^{\beta+3}\,\varphi\,\sqrt{10}
	>10^{\beta+3}\cdot 37\cdot\sqrt{10}.
	$$
	Then, $\delta(p_{n+3})\ge \delta(10^{\beta+3}\cdot 37\cdot\sqrt{10})=\beta+5$. Hence, $\delta(p_{n+3})=\beta+5$.
	Thus, we conclude $\delta(p_{n+3})=\beta+5=\delta(q_{n+3})+1$.
\end{proof}

\section{Conclusion}

We summarize here our enquires on the WhtasApp question raised in the introduction. 
The set $\mathcal{S}$ of positive solutions of the diophantine equation $x(x+1)=10y(y+1)$ 
is the set of terms of the sequence 
$(x_n,y_n)$ defined by
$$
(x_1,y_1)=(4,1),\quad (x_2,y_2)=(20,6),\quad (x_3,y_3)=(39,12),
$$
and, for $n\ge 1$,
$$
x_{n+3}=19x_n+60y_n+39, \qquad y_{n+3}=6x_n+19y_n+12.
$$
Explicitly, for $k\in\{1,2,3\}$, if
$$
A_k=\frac{1}{2}\left(\frac{x_k}{\sqrt{10}}+y_k+\frac{10+\sqrt{10}}{20}\right),\ 
\mbox{ and }\ \varphi=19+6\sqrt{10},
$$ 
then, for $n\ge 1$,
$$
x_{3n+k}=\lfloor A_k\sqrt{10}\varphi^n\rfloor,\qquad y_{3n+k}=\lfloor A_k\,\varphi^n\rfloor.
$$
The set $\mathcal{C}$ of solutions that concatenate well with regard to our WhastApp problem is formed by the terms $(x_n,y_n)\in\mathcal{S}$ such that $x_n$ has one more digit than $y_n$. The set $\mathcal C$ has infinitely many elements, and the sequence of quotients $(y_n+1)/(x_n+1)$ 
with $(x_n,y_n)\in\mathcal{C}$ is strictly decreasing with limit $1/\sqrt{10}$. 
\bigskip

\noindent\textbf{Acknowledgement}\\
	We would like to thanks Xavier de Cabo, the sender of the WhatsApp message.


\begin{thebibliography}{[0}
	\bibitem{NivZucMon}
	I.~Niven,  H.~S.~Zuckerman, and H.~L.~Montgomery, \emph{An Introduction to the Theory of Numbers}, John Wiley \&  Sons, 1991.
	
	\bibitem{Ono} 
	T.~Ono, \emph{An Introduction to Algebraic Number Theory}, Plenum Press, 1990.
	
	\bibitem{Stevenhagen} 
	P. Stevenhagen, \emph{Number Rings}, \\
	\texttt{http://websites.math.leidenuniv.nl/algebra/ant.pdf} 
	
\end{thebibliography}
\end{document}